\documentclass[11pt,a4paper]{amsart}
\usepackage{amssymb,amsmath}
\usepackage{comment}
\usepackage{etoolbox}
 \usepackage{lipsum}
\apptocmd{\sloppy}{\vbadness 10000\relax}{}{}

\def\({\left(}
\def\){\right)}

\def\Cal{\mathcal}
\def\Ree{\operatorname{Re}}

\def\eb{\varepsilon}

\def\R {\mathbb{R}}

\newcommand{\be}{\begin{equation} }
\newcommand{\ee}{\end{equation} }

\def \and{\qquad\text{and}\qquad}

\def\Bbb{\mathbb}
\def\Dt{\partial_t}
\def\Dx{\Delta_x}

\def\({\left(}
\def\){\right)}

\def\eb{\varepsilon}
\def\Cal{\mathcal}
\def\eb{\varepsilon}

\def\R {\mathbb{R}}

\def\<{\left<}
\def\>{\right>}

\def \and{\qquad\text{and}\qquad}

\def\Bbb{\mathbb}
\def\Dt{\partial_t}
\def\Dx{\Delta_x}

\newtheorem{proposition}{Proposition}[section]
\newtheorem{theorem}[proposition]{Theorem}
\newtheorem{corollary}[proposition]{Corollary}
\newtheorem{lemma}[proposition]{Lemma}
\theoremstyle{definition}
\newtheorem{definition}[proposition]{Definition}
\newtheorem{remark}[proposition]{Remark}

\numberwithin{equation}{section}

\def\be{\begin{equation}}
\def\ee{\end{equation}}
\def\bp{\begin{proof}}
\def\ep{\end{proof}}

\def \no#1#2#3 {{\bf #1} (#3), #2.}
\def \eds#1#2#3 {#1, #2, #3.}

\title[Inertial Manifolds for complex Ginzburg-Landau equations]
{Inertial manifolds for  3D complex Ginzburg-Landau equations with periodic boundary conditions}
\author[A. Kostianko, C. Sun,  and  S. Zelik]
{ Anna Kostianko${}^{1,3}$, Chunyou Sun${}^1$, and Sergey Zelik${}^{1,2,4}$}
\address{${}^1$ \phantom{e}School of Mathematics and Statistics, Lanzhou University, Lanzhou  \\ 730000,
P.R. China}
  
\address{${}^2$ University of Surrey, Department of Mathematics, Guildford, GU2 7XH, United Kingdom.}
\address{${}^3$  Imperial College, London SW7 2AZ, United Kingdom.}
\address{${}^4$ Keldysh Institute of Applied Mathematics, Moscow, Russia.}
\email{aNNa.kostyanko@mail.ru (A. Kostianko)}
\email{sunchy@lzu.edu.cn (C. Sun)}
\email{s.zelik@surrey.ac.uk (S. Zelik)}
\thanks{This work is partially supported by  the RSF grant  19-71-30004 (Russia), the NSFC Grants no. 11522109 and 11871169 (China)
 and the Leverhulme grant No. RPG-2021-072 (United Kingdom).\newline\phantom{eggog}\newline\phantom{eggog}}

\begin{document}
 \begin{abstract} We prove the existence of an Inertial Manifold for 3D complex Ginzburg-Landau equation with periodic boundary conditions as well as for more general cross-diffusion system assuming that the dispersive exponent is not vanishing. The result is obtained under the assumption that the parameters of the equation is chosen in such a way that the finite-time blow up of smooth solutions does not take place. For the proof of this result we utilize the method of spatio-temporal averaging recently suggested in \cite{K20}.
\end{abstract}

\maketitle
\tableofcontents
\section{Introduction}\label{s0}
It is believed that the  dynamics generated by a dissipative partial differential equation (PDE) at least in a bounded domain can be described,
up to some "non-essential" transient effects, by finitely many parameters, the so-called {\it order}
parameters in the terminology of I. Prigogine, see \cite{Pri77}. Ideally, it is expected that these
order parameters obey a system
of ordinary differential equations (ODEs) which is called an inertial form (IF) of the initial
dissipative system. Thus, the IF if it exists allows us to reduce the study of the
dynamics generated by PDEs to the study of the corresponding system of ODEs. The dream to understand the nature of turbulence using the ideas and methods of classical dynamics permanently inspires the development of the dynamical theory of dissipative systems during the last 50 years, see \cite{BV92,CV02,Fef06,FP67,F95,Lions1969,R01,T95,T97} and references therein.
\par
However, despite a lot of progress done by prominent researches, the nature of the above mentioned
finite-dimensional reduction and its rigorous justification somehow remains a mystery. Moreover, as recent examples
 and counterexamples show, there are deep obstacles to effective realization of this program, e.g. related
  with the {\it smoothness} of the IF and related finite-dimensional reduction, see \cite{Z14}
  and references therein.
\par
Indeed,  IFs with H\"older continuous vector fields can be obtained under weak assumptions on the considered dissipative system using the concept of a global attractor.  
By definition, a global attractor of a dynamical system (DS) is a compact
  invariant set in the phase space which attracts the images of bounded sets as time tends to infinity.
The main achievement of the attractors theory is that a global attractor  exists for many classes of dissipative PDEs arising in applications  and usually  has
finite Hausdorff and box-counting dimensions, see  \cite{BV92,CV02,R01,T97,MZ08} and references
 therein. The class of
 such systems includes reaction-diffusion and 2D Navier-Stokes systems, pattern formation equations
 (like Cahn-Hilliard or Swift-Hohenberg ones),
  damped wave equations and many others. Then the associated IF  can be constructed via the Man\'e projection theorem, see \cite{R11}
    and references therein. Under this approach the box-counting dimension of the attractor $\Cal A$ is usually interpreted as a number of "degrees of freedom" in the reduced IF. In particular, this explains the permanent interest to various upper and lower bounds for the box-counting dimension of a global attractor.
      \par
      However,  the described scheme leads to a drastic and unacceptable loss of regularity.
      Indeed, on the one hand, it is not clear whether and in what aspects the study of the reduced H\"older continuous IF is simpler than the direct study of the initial smooth PDE and, on the other hand, there is no way in general to increase the regularity of the IF. In a fact, there are natural examples where the box-counting dimension of the attractor is low (e.g, 3), but a Lipschitz IF does not exist. Moreover, the reduced  dynamics on the attractor contains features which can hardly be interpreted as "finite-dimensional" (like limit cycles with super-exponential rate of attraction, traveling waves in Fourier space, etc.),
         see \cite{EKZ13,MPSS93,Rom00,KZ18,Z14} for more details. In these cases, the "finite-dimensionality" obtained
          via Man\'e projections looks artificial and controversial and it seems more natural to
           accept that the dynamics here is infinite-dimensional despite the finiteness
            of box-counting dimension.
\par
An alternative approach to the finite-dimensional reduction problem is based on the concept of an {\it inertial}
  manifold (IM) suggested in \cite{FST88}. Roughly speaking, an IM $\Cal M$ is a sufficiently smooth
   (at least Lipschitz)
   finite-di\-men\-sio\-nal invariant submanifold of the phase space which is normally-hyperbolic
   and exponentially stable. If such an object exists, then the finite-dimensional reduction is
    ideally justified. Indeed, the reduction of the initial PDE to the manifold $\Cal M$ gives us the
     desired IF and the normal hyperbolicity gives us the so-called asymptotic phase or exponential tracking  property which in turn gives us a nice rigorous interpretation in what sense the
      transient features are "non-essential".
\par
Unfortunately, being a sort of a center manifold, an IM requires strong separation of the phase space on slow and fast variables which is usually stated in the form of {\it spectral gap} conditions or/and
 invariant cone properties, see \cite{CL02,CFNT89,FST88,FNST88,Fen72,HGT15,M91,MS89,R94,RT96} and references therein for more details. These spectral gap assumptions give a severe restriction on the class of dissipative systems considered, for instance, even for the simplest class of reaction-diffusion equations, they are satisfied for the case of one spatial variable only, so many important equations including 2D Navier-Stokes equations remain out of reach of this theory. It is also known that the spectral gap conditions are sharp on the level of abstract semilinear equations, so there is almost no hope to go further remaining in the class of abstract functional models. However, the situation becomes better if the concrete class of equations is considered, in particular, the structure of the non-linearity (e.g. it locality) may be helpful for relaxing the spectral gap condition.
 \par
 The most famous example is the so-called spatial averaging method introduced by G. Sell and J. Mallet-Paret, see \cite{M-PS88} (and also \cite{Kwe99} for some extensions) which allowed to construct an IM for a scalar reaction-diffusion equation in 3D with periodic boundary conditions. This  method has been later simplified/extended to a number of other equations, such as Cahn-Hilliard equations, see \cite{KZ15} and  various simplified versions of Navier-Stokes equations, see \cite{K18,GG18,LS20}, see also \cite{FCH} for a unified treatment of these and similar equations. An alternative method which has been initially suggested in \cite{Kwak} in the erroneous construction of an IM for 2D Navier-Stokes equations consists of transforming the initial equations or/and  embedding them to a bigger system of equations for which the spectral gap conditions are satisfied. This method has been recently applied for clarifying the situation with existence or non-existence of IMs for 1D reaction-diffusion-advection problems, see \cite{KZ18,KZ17}.
 \par
 The main aim of the present paper is to verify the existence of an IM for the following cross-diffusion system in 3D: 
 \begin{equation}\label{0.main}
 \Dt u-(1+i\omega)\Dx u=f(u,\bar u)
 \end{equation}
 endowed with periodic boundary conditions.  Here, $u = u_{Re}(t, x) + i u_{Im}(t, x)$ is an unknown complex-valued
function, $\omega\in\R$, $\bar u=u_{Re}-iu_{Im}$ is a complex conjugate function and $f$ is a given smooth function. In the particular case
\begin{equation}\label{0.cGL}
f(u,\bar u)=(1+i\beta)u-(1+i\gamma )u|u|^2
\end{equation}
we end up with the classical Ginzburg–Landau equation (see \cite{DGL,Mie} and references therein for
more details concerning this equation and its physical meaning).
\par
As usual, in order to speak about IMs, we first need to get the dissipativity of the considered PDE. Unofortunately the sharp conditions which guarantee such a dissipativity are not known for system \eqref{0.main} even in the case of Ginzburg-Landau nonlinearity. Moreover there is an evidence that solutions with smooth initial data can blow up in finite time despite the immediate dissipative estimate in $L^2$-norm in the so-called self-focusing case $\omega\gamma<0$, see \cite{budd}, see also \cite{DGL} for a list of sufficient conditions for dissipativity in higher norms. Note also that, due to the dissipative estimate in $L^2$ norm, we have the following dichotomy at least in the case of classical Gingburg-Landau equation: either the problem is globally solvable and dissipative in higher energy norms (e.g. in the $H^2$-norm) or there are solutions with smooth initial data which blow up in finite time. In the first case, we have an absorbing ball in $H^2\subset C$ and can cut-off the nonlinearity outside this ball making it globally bounded.
\par
Since an accurate study of the conditions which guarantee the global well-posedness and dissipativity of problem \eqref{0.main} is out of scope of this paper, we assume from the very beginning that this dissipativity is already established and the corresponding cut-off procedure is done and the modified nonlinearity $f\in C_0^\infty(\Bbb C)$ (we also assume for simplicity that $f$ is $C^\infty$-smooth although the finite smoothness is enough for all our results).   
\par
The main result of the paper is the following theorem.
\begin{theorem} Let $\omega\ne0$ and $f\in C_0^\infty(\Bbb C)$. Then problem \eqref{0.main} possesses an IM of the finite dimension (see section \ref{s3} for the rigorous definition) which is $C^{1+\eb}$-smooth for some $\eb>0$.
\end{theorem}
The proof of this result utilizes the so-called method of spatio-temporal averaging suggested in \cite{K20}
for constructing Bi-Lipschitz Man\'e projectors for the attractors of equations of the form \eqref{0.main}. This method is based on a combination of  the spatial averaging of G. Sell and J. Mallet-Paret (which allows to replace the terms $f'_u(\cdot)v$ and $f'_{\bar u}(\cdot)\bar v$ in the corresponding equation of variations by their spatial averaging $\<f'_u(\cdot)\>v$ and $\<f'_{\bar u}(\cdot)\>\bar v$ respectibely) and the classical temporal averaging of rapid oscillations arising due to
the presence of a dispersive term $i\omega\Dx v$ (which finally allows us to kill the term  $\<f'_{\bar u}(\cdot)\>\bar v$), see section \ref{s3} for more details. 
\par
Note that only spatial averaging is not sufficient here exactly due to the presence of an extra term $\<f'_{\bar u}(\cdot)\>\bar v$ which prevents us to get the desired  dichotomy and construct an IM (as well-known, the spatial averaging method does not work in general for systems, only for scalar equations), so the presence of the second (temporal) averaging is crucial. Moreover, as shown in \cite{Rom00}, the IM may not exist for equation \eqref{0.main} in the case $\omega=0$ when this second averaging is impossible (at least under some natural extra normal hyperbolicity assumptions).
\par
In the present paper we lift the result of \cite{K20} from the level of Bi-Lipschitz Man\'e projectors to IMs by adapting the method of spatio-temporal average to the invariant cone technique. In contrast to the previous works, see \cite{Z14} and references therein, we have to use the so-called "floating" cones which depend on the point of the phase space. Although such a dependence is not surprising for the classical (finite-dimensional) theory of normally-hyperbolic invariant manifolds, to the best of our knowledge it has been never appeared before in the theory of IMs.
\par
To conclude we note that complex Ginzburg-Landau equation is also widespread in the theory of turbulence. In particular, it is used to indicate the difference between the so-called soft and hard turbulence, see \cite{Gib} and references therein. Our result shows that IMs exist in both cases (if the dispersive exponent $\omega$ does not vanish) if the considered system remains dissipative and the finite-time blow up is the only mechanism which may destroy the existencs of IM.
 \par
 The paper is organized as follows. In section \ref{s1} we collect some preliminary results which will be used throughout of what follows. In particular, we introduce and discuss here  mainly following \cite{FCH} and \cite{K18} the further cut-off procedure for the nonlinearity $f$ which is necessary to deal with IMs. Note that this procedure is a bit more delicate than used before since we need to control the time derivatives of solutions in order to do temporal averaging and is actually one of the key technical tools to get an IM. The invariant cones for equation \eqref{0.main} are constructed in section \ref{s2} and the existence of an IM is derived based on these invariant cones in section \ref{s3}. The extra smoothness of the constructed IM is also verified there.

 \section{Notations and preliminaries}\label{s1}
In this section, we introduce useful notations and prepare some technical tools which will be used throughout the paper. We start with the necessary spectral projectors.
\par
We denote by $W^{s,p}(\Bbb T)$ the Sobolev space of distributions on a torus $\Bbb T:=[-\pi,\pi]^3$ whose derivatives up to order $s$ belong to $L^p(\Bbb T)$. In the Hilbert case $p=2$, we use the notation $H^s(\Bbb T):=W^{s,2}(\Bbb T)$, see e.g. \cite{Tri78} for more details.
\par
 The eigenvalues of the Laplacian on the torus are naturally parameterized by triples of natural numbers $n:=(k,l,m)\in\Bbb Z^3$
$$
\lambda_{n}:=|n|^2=k^2+l^2+m^2
$$
with the corresponding eigenvectors $e_n(x):=e^{in.x}$. Then any function $u\in D'(\Bbb T)$ can be presented  as Fourier series:
$$
u(x)=\sum_{n\in\Bbb Z^3}u_n e_n,\ \ u_n:=\frac1{(2\pi)^3}(u,e_n),
$$
where $(u,v):=\int_{\Bbb T}u(x)\bar v(x)\,dx$ is a standard inner product in $L^2(\Bbb T)$. In particular, by Parseval equality,
$$
\|u\|^2_{H^s}=(2\pi)^3\sum_{n\in\Bbb Z^3}(1+|n|^2)^{s}|u_n|^2,\ \ s\in\R.
$$
For every $N\in\Bbb N$ and $0<K<N$, we define the spectral projectors
$$
P_Nu:=\sum_{1+|n|^2\le N}u_ne_n,\ \ Q_Nu:=(1-P_N)u=\sum_{1+|n|^2>N}u_ne_n
$$
as well as
$$
\Cal I_{N,K}u:=\sum_{N-K<1+|n|^2<N+K}u_ne_n,\ \ \Cal P_{N,K}=P_N\circ (1-\Cal I_{N,K}),
$$
and $\Cal Q_{N,K}:=Q_N\circ(1-\Cal I_{N,K})$. Then, keeping in mind the spatial averaging principle, we split any function $u$ in a sum of lower ($u_+:=\Cal P_{N,K}u$),
 intermediate ($u_I:=\Cal I_{N,K}u$) and higher ($u_-:=\Cal Q_{N,K}u$) modes:
 $$
 u=u_++u_I+u_-.
 $$
 We now turn to our complex Ginzburg-Landau equation in the form
 \begin{equation}\label{1.cgl}
 \Dt u+(1+i\omega)Au=f(u,\bar u),\ \ Au:=(1-\Dx)u
 \end{equation}
 on a torus $x\in\Bbb T$. Here $u=u_1(t,x)+iu_2(t,x)$ is an unknown complex valued function,
 $\bar u=u_1-iu_2$ is complex conjugate, $f$ is a given smooth nonlinearity with finite support: $f\in C_0^\infty(\R^2,\R^2)$ and $\omega\in\R$, $\omega\ne0$ is a given number. For simplicity we also assume that $f(0,0)=0$. Thus, we  assume from the very beginning that the initial nonlinearity in complex Ginzburg-Landau equation is cut off outside of the attractor/absorbing set. It is immediate to see that problem \eqref{1.cgl} is well-posed and generates a dissipative semigroup $S(t)$ in $H$. Moreover, for any $s\in\R_+$, there exists an absorbing ball $\Cal B_s:=\{u\in H^s,\ \|u\|_{H^s}\le R_s\}$ for this semigroup. The latter means that for every bounded set $B\subset H$ there exists $T=T(B,s)$ such that
  \begin{equation}
  S(t)B\subset \Cal B_s,\ \ \ \forall t\ge T,
  \end{equation}
  see e.g. \cite{BV92,T95,Z14}.
  \par
  However, this is not enough for constructing an inertial manifold (to be precise, we do not know whether or not \eqref{1.cgl} possesses an IM) and some more tricky cut off of the nonlinearity is necessary. Note that in order to be consistent with equation \eqref{1.cgl}, we only need to guarantee that our cut off procedure does not change the nonlinearity inside of the absorbing ball $B_s$ for some fixed $s$.
   \par
   Following the approach developed in \cite{K18, FCH}, we introduce the truncated function $W(u)$ as follows.
 \par
  Let $\phi\in C^\infty(\Bbb C)$ be such that $\bar \phi( z)=\phi(\bar z)$, $\phi(z)=z$ for $|z|\le 1$ and $\phi(z)=0$ for $|z|\ge 2$.
        Then for a
       given positive constant $C_*$ and sufficiently large exponent $s$, we define the function
       $W: H\to H$ via
       \begin{equation}\label{1.W}
        W(u)=\sum_{n\in\Bbb Z^3} C_*(1+|n|^2)^{-s/2}\phi\(\frac{(1+|n|^2)^{s/2}(u,e_n)}{C_*}\)e_n.
       \end{equation}
The elementary properties of this truncation function are collected in the following proposition.
\begin{proposition}\label{Prop1.W} Let the function $W$ be defined via \eqref{1.W}. Then,
\par
1. The map $W$ is bounded and continuous as a map from $H$ to $H^{s_0}$, where $s_0>0$ is such that
\begin{equation}\label{1.sm-eat}
\sum_{n\in\Bbb Z^3}(1+|n|^2)^{s_0-s}<\infty.
\end{equation}
\par
2. $W(u)\equiv u$ if $u\in H^s$ and $\|u\|_{H^s}\le C_*^2$.
\par
3. The function $W$
 is Hadamard differentiable as a map from $H$ to $H$ and the derivative is given by
\begin{equation}\label{1.Wdif}
W'(u)v=\sum_{n\in\Bbb Z^3} \phi'\((1+|n|^2)^{s/2}\frac{(u,e_n)}{C_*}\)(v,e_n)e_n.
\end{equation}
4. There exists a positive constant $C$ such that, for every $\kappa\in \R$
\begin{equation}\label{1.W-sm}
\|W'(u)\|_{\Cal L(H^\kappa,H^\kappa)}\le C,\  \ \|W'(u_1)-W'(u_2)\|_{\Cal L(H^\kappa,H^\kappa)}\le C\|u_1-u_2\|_{H^{s}}.
\end{equation}
for all $u,u_1,u_2\in H^s$.
\end{proposition}
The proof of this proposition is given in \cite{FCH}, see also \cite{K18}.
\par
From now on we fix some $\frac32<s_0<2$, $s_0+\frac32<s<4$ and take  $C_*:=\sqrt{R_s}$ in order to guarantee that $W(u)=u$ for all $u\in\Cal B_s$. In addition, we will have that $W(u)\in H^{s_0}\subset C(\Bbb T)$ due to the Sobolev embedding.
\par
The basic idea is to replace $f(u,\bar u)$ by $f(W(u),\bar W(u))$ in equation \eqref{1.cgl},  but analogously to \cite{M-PS88} and \cite{FCH}, we need some extra terms to implement the spatial averaging method.
Namely, following  \cite{M-PS88}, we introduce a cut-off function $\varphi(z)$ which equals to $0$ for $z\le R_1^2$ (where $R_1$ is the radius of the absorbing ball $\Cal B_1$)
 and equals to $-1/2$ if
  $z\ge \widetilde R^2$ for some $\widetilde R>R_1$. Then, we define the map $T=T_N:H\to H$ via
\begin{equation}\label{5.T}
T(u):=\varphi(\|A^{1/2}P_Nu\|^2_H)A^{1/2}P_Nu.
\end{equation}
The key property of this map is stated in the following lemma.
\begin{lemma}\label{Lem1.T} It is possible to fix the cut-off function $\varphi$ in such a way that
\begin{equation}\label{1.T-good}
(T'(u)v,v)\le \frac12\((N^{1/2}-A^{1/2})v,v\),\ \ v\in P_NH
\end{equation}
and $(T'(u)v,v)=-\frac12\|P_Nv\|^2_{H^{1/2}}$ if $\|P_Nu\|_{H^1}\ge \widetilde R$.
\end{lemma}
The proof of this lemma is given in \cite{M-PS88} (see also \cite{Z14}).
\par
Note that
\begin{multline*}
T'(u)v=\varphi(\|A^{1/2}P_Nu\|^2_{H})A^{1/2}v+\\+2\varphi'(\|A^{1/2}P_Nu\|^2_H)\Ree(A^{1/2}u,A^{1/2}v)A^{1/2}P_Nu
\end{multline*}
for $v\in P_NH$ and the norm of the operator $T'(u)$
depends on $N$. Namely, it is not difficult to show using the fact that $\varphi'$ has a finite support that
\begin{equation}\label{1.T-est}
\|T'(u)v\|_{H}\le C N^{1/2}\|v\|_H,
\end{equation}
where the constant $C$ is independent of $u$.
\par
We also introduce the spatial average operator $a(u)$ as follows:
\begin{multline}\label{1.a-a}
a(u)v:=a_u(u)v+a_{\bar u}(u)\bar v,\\ a_u(u):=\<f'_u(u(\cdot),\bar u(\cdot))\>,\ \
 a_{\bar u}(u):=\<f'_{\bar u}(u(\cdot),\bar u(\cdot))\>,
\end{multline}
where $\<w\>:=\frac1{(2\pi)^3}\int_{\Bbb T}w(x)\,dx$ is the mean value of $w$ on a torus.
\par
Finally, we fix one more smooth cut-off function $\theta(z)$ which equals to one if $z\le  R_0^2$ and
 zero if $z>4R_0^2$, where $R_0$ is the radius of the absorbing ball $\Cal B_0$ in $H$ and define
\begin{multline}\label{1.comp}
F(u):=f(W(u),\bar W(u))-a(W(u))W(u)+\\+\theta(\|u\|^2_H)a(W(u))u-T_N(u).
\end{multline}
Then, due to Proposition \ref{Prop1.W}, the function $F$ will be bounded and continuous as the map
  from $H$ to $H$ and its Gateaux derivative will have the form
  \begin{multline}\label{1.derder}
F'(u)v=[f'_u(W(u),W(\bar u))W'(u)v+f'_{\bar u}(W(u),W(\bar u))W'(\bar u)\bar v-\\-a_u(W(u))W'(u)v-a_{\bar u}(W(u))W'(\bar u)\bar v]+\\+[\theta(\|u\|^2_H)(a_u(W(u))v+a_{\bar u}(W(u))\bar v)]-\\-a'(W(u))[W'(u)v,W(u)]+
[2\theta'(\|u\|^2_H)(u,v)a(W(u))+\\+\theta(\|u\|^2_H)(a'(W(u)),W'(u)v)]u-T'_N(u)v=\\=
l_1(u)v+l_2(u)v+l_3(u)v+l_4(u)v-T'_N(u)v.
  \end{multline}
  Indeed, the verification of \eqref{1.derder} is straightforward (see \cite{FCH}), so we left the details to the reader.
   \par
Note that only the term $T_N(u)$ depends explicitly on $N$ now, so the norms of all other terms
 are independent of $N$. In particular, since $Q_NT_N(u)\equiv 0$, we have that
 \begin{equation}\label{1.q-est}
\|Q_NF(u)\|_{H}\le C,\ \ u\in H,
 \end{equation}
 where $C$ is independent of $N$.
 \par
 Thus, we replace equation \eqref{1.cgl} by the modified one
 \begin{equation}\label{1.cglm}
 \Dt u+(1+i\omega)Au=F(u),
 \end{equation}
where $F(u)$ is defined by  \eqref{1.comp} and will construct an IM for this equation only. The construction of the non-local nonlinearity $F(u)$ guarantees that $F(u)=f(u,\bar u)$ if $u$ belongs to the absorbing set $\Cal B_0\cap\Cal B_1\cap\Cal B_s$, so such a manifold will capture all the nontrivial limit dynamics of \eqref{1.cgl}.
 \par
 To conclude this section, we prepare some more technical tools which will be used later.
 \begin{lemma}\label{Lem5.Q} Let the estimate \eqref{1.q-est} hold. Then, for any $\kappa>0$, the $Q_N$-component of
the solution $u(t)$ of problem \eqref{1.cglm} possesses the following estimate:
\begin{equation}\label{1.q-dis}
\|Q_Nu(t)\|_{H^{2-\kappa}}\le C_1\frac{1+t^M}{t^M}e^{-\beta t}\|Q_Nu(0)\|_{H^{-\gamma}}+C_2,
\end{equation}
where the constants $M,\beta,\gamma>0$ and $C_1,C_2$ are independent of $N$ and $u$ (but may depend on $\kappa$).
\end{lemma}
 Indeed, estimate \eqref{1.q-dis} follows from \eqref{1.q-est} and the parabolic regularity estimates applied to the equation
\begin{equation}\label{1.par-reg}
\Dt Q_Nu+(1+i\omega)AQ_Nu=Q_NF(u),
\end{equation}
see \cite{FCH}.
 \par
 As usual, estimate \eqref{1.q-dis} allows us to verify the strong cone property (which is the crucial step in the construction of an IM)  for the trajectories $u(t)$ satisfying
\begin{equation}\label{1.QIM}
\|Q_Nu(t)\|_{H^{2-\kappa}}\le C_2, \ t\ge0
\end{equation}
only, see \cite{FCH,M-PS88,KZ15,Z14} and \S\ref{s3} below.
 \par
 Note that in contrast to the $Q_N$-component of $u(t)$, the $P_N$-component is typically unbounded
  on the IM, so we cannot assume any uniform bounds for it. Instead, we will use the extra
   map $T$ and Lemma
   \ref{Lem1.T} in order to control it.
 \par
 Unfortunately, the parabolic regularity estimates are not strong enough to derive estimate  \eqref{1.q-dis} with $\kappa=0$ from \eqref{1.q-est} (and, in contrast to the case of only spatial averaging considered in \cite{FCH}, we need the estimate with $\kappa=0$ in order to control the $H$ norm of $Q_N\Dt u$ which in turn is necessary for temporal averaging), so we need some bootstrapping arguments. Namely, using the facts that $W(u)$ is bounded in $H^{s_0}$ and that $H^{s_0}$ is an algebra (since $s_0>\frac32$), we establish that
 \begin{equation}\label{1.q-good}
 \|Q_N F(u)\|_{H^{s_0}}\le C_1+C_2\|Q_Nu\|_{H^{s_0}}
 \end{equation}
 where the constants $C_1$ and $C_2$ are independent of $N$ and the linear term in the RHS of \eqref{1.q-good} comes from
 $$
 \|Q_N\theta(\|u\|^2_H)a(W(u))u\|_{H^{s_0}}\le C_2\|Q_Nu\|_{H^{s_0}}.
 $$
 The parabolic regularity estimate applied to \eqref{1.par-reg} together with \eqref{1.q-good} and \eqref{1.q-dis} give us the analogue of estimate \eqref{1.q-dis} where $2-\kappa$ is replaced by $2+s_0-\kappa$. This in turn means that we can verify the cone property under stronger (than \eqref{1.QIM}) assumption that
\begin{equation}\label{1.QIMs}
\|Q_Nu(t)\|_{H^{2+s_0-\kappa}}+\|Q_N\Dt u(t)\|_{H^{s_0-\kappa}}\le C_2, \ t\ge0.
\end{equation}
This improved estimate is necessary to control the time derivative of the term
\begin{equation}\label{1.cbar}
\Cal C_{\bar u}(u):=\theta(\|u\|^2_H)a_{\bar u}(W(u))
\end{equation}
appearing in  transformations related with temporal averaging. Namely, we have the following result.

\begin{lemma}\label{Lem1.temp} Let the assumption \eqref{1.QIMs} hold. Then
\begin{equation}\label{1.C-est}
|\Dt \Cal C_{\bar u}(u(t))|\le C\(N^{1/2}+N\chi_{\|P_Nu(t)\|_{H^1}\ge4\widetilde R}(t)\),
\end{equation}
where the constant $C$ is independent of $N$ and $\chi_V(t)$ is a characteristic function of the set $V$.
\end{lemma}
\begin{proof} Indeed, let $t$ be such that $\|P_N u(t)\|_{H^1}\ge4\widetilde R$. Then, using that the function $f$ is smooth and has a finite support and that $W'(u)$ is a bounded operator, we get the estimate
\begin{equation}\label{1.cest}
|\Dt \Cal C_{\bar u}(u(t))|\le \widetilde\theta(\|u(t)\|_H^2)\|\Dt u(t)\|_H,
\end{equation}
where the function $\widetilde\theta(z)$ vanishes for $z>4R_0^2$. On the other hand, from \eqref{1.QIMs} we know that $\|Q_N\Dt u(t)\|_H$ is uniformly bounded. To estimate $\Dt P_N u$, we apply the orhoprojector $P_N$ to equation \eqref{1.cglm} to get
\begin{multline}\label{1.eggog}
\|P_N\Dt u\|\le C(1+\|P_N u\|_{H^2})\le \\\le C(1+N \|P_Nu\|_H)\le C(1+N\|u\|_H).
\end{multline}
Therefore,
$$
\|\Dt u(t)\|_H\le C(1+N\|u(t)\|_H)
$$
and inserting this estimate into \eqref{1.cest}, we get the desired result.
\par
Let now $\|P_N u(t)\|_{H^1}\le 4\widetilde R$. Then,
$$
\|P_Nu(t)\|_{H^2}\le N^{1/2}\|P_Nu(t)\|_{H^1}\le 4N^{1/2}\widetilde R
$$
and from \eqref{1.eggog}, we see that
$$
\|\Dt u(t)\|_H\le CN^{1/2}
$$
which gives \eqref{1.C-estd}
 and finishes the proof of the lemma.
\end{proof}
The last lemma in this section gives the analogue of estimate \eqref{1.C-est} for a convex sum of solutions of equation \eqref{1.cglm} and will be useful for establishing the cone property for differences of solutions of \eqref{1.cglm}.
\begin{lemma}\label{Lem1.temp1} Let $u_1(t)$ and $u_2(t)$ be two solutions of \eqref{1.cglm} satisfying \eqref{1.QIMs} and let $\alpha\in[0,1]$. Then, the following estimate holds:
\begin{multline}\label{1.C-estd}
|\Dt \Cal C_{\bar u}(\alpha u_1(t)+(1-\alpha)u_2(t))|\le \\\le C\(N^{1/2}+N\(\chi_{\|P_Nu_1(t)\|_{H^1}\ge4\widetilde R}(t)+\chi_{\|P_Nu_2(t)\|_{H^1}\ge4\widetilde R}(t)\)\),
\end{multline}
where the constant $C$ is independent of $N$, $\alpha$, $u_1$ and $u_2$.
\end{lemma}
\begin{proof} Arguing as in the proof of Lemma \ref{Lem1.temp}, we get the estimate
\begin{multline}\label{1.eggog1}
|\Dt \Cal C_{\bar u}(\alpha u_1(t)+(1-\alpha)u_2(t))|\le \\
\widetilde\theta(\|\alpha P_N u_1(t)+(1-\alpha)P_N u_2(t)\|_H^2)
(1+\|\alpha P_N\Dt u_1(t)+(1-\alpha)P_N\Dt u_2(t)\|_H)
\end{multline}
for some smooth function $\widetilde\theta(z)$ vanishing if $z>4R_0^2+C$ (here we have implicitly used that the $H$-norms of $Q_Nu_i$ are bounded). Moreover, from equation \eqref{1.cglm}, we infer that
\begin{multline}\label{1.huge}
\|\Dt P_N(\alpha  u_1+(1-\alpha)u_2)\|_H\le C(1+N\|P_N(\alpha  u_1+(1-\alpha)u_2)\|_H)\\+
N^{1/2}\|\alpha \varphi(\|A^{1/2}P_Nu_1\|_H^2)P_Nu_1+(1-\alpha)\varphi(\|A^{1/2}P_Nu_2\|_H^2)P_Nu_2\|_H,
\end{multline}
where we have used again that all parts of the nonlinearity $F$ except of the term  $T(u)$ are uniformly bounded.
\par
Let us consider now three  alternative cases.
\par
{\it Case I.} Both functions $u_1,u_2$ satisfy
$$
\|A^{1/2}P_Nu_i\|_H\ge 4\widetilde R.
$$
In this case, both $\varphi$'s in the second term are equal to $-\frac12$. Therefore, we have
$$
\|\Dt P_N(\alpha u_1+(1-\alpha)u_2\|_H\le C(1+N\|P_N(\alpha u_1+(1-\alpha)u_2)\|_H)
$$
and inserting this estimate to the RHS of \eqref{1.eggog1}  gives the desired inequality.
\par
{\it Case II.} One of the functions  $u_1(t),u_2(t)$ (say, $u_1$) satisfies
$$
\|P_NA^{1/2}u_1(t)\|_H\le 4\widetilde R
$$
and $\|P_NA^{1/2}u_2(t)\|_H\ge 4\widetilde R$.
In this case $\|u_1\|_H$ is uniformly (with respect to $N$) bounded and, therefore, \eqref{1.huge} reads
$$
\|\Dt P_N(\alpha u_1+(1-\alpha)u_2)\|_H\le C'N(1+\|(1-\alpha)P_N u_2\|_H)
$$
for some  constant $C'$ independent of $N$. Inserting this estimate to the RHS and using that
$$
\|P_N(\alpha u_1+(1-\alpha)u_2\|_H\ge\|(1-\alpha)P_Nu_2\|_H-C,
$$
we get the desired estimate in this case as well.
\par
{\it Case III.} Both functions $u_1(t),u_2(t)$ satisfy
$$
\|A^{1/2}P_Nu_i(t)\|_{H}\le 4\widetilde R.
$$
Then, as in Lemma \ref{Lem1.temp}, $\|P_Nu_i\|_{H^2}\le 4N^{1/2}\widetilde R\le CN^{1/2}$ and, therefore,
$$
\|\alpha \Dt u_1+(1-\alpha)\Dt u_2\|_{H}\le CN^{1/2}
$$
which together with \eqref{1.eggog1} gives the desired result.
\par
Thus, estimate \eqref{1.C-estd} is verified in all 3 cases and the lemma is proved.
\end{proof}

\section{Spatio-temporal averaging and cone property}\label{s2}
The aim of this section is to verify the strong cone property for solutions of the modified problem \eqref{1.cgl}. We start with the equation of variations associated with it:
\begin{equation}\label{2.var}
\Dt v+(1+i\omega)Av=F'(u(t))v,
\end{equation}
where $F'(u(t))$ is defined by \eqref{1.derder} and the trajectory $u(t)$ is assumed to satisfy \eqref{1.QIMs}.  In contrast to the previous works on inertial manifolds, the invariant cones for \eqref{2.var} will depend on $t$ through the trajectory $u(t)$ and in order to introduce them we need some preparations.
\par
First, we split the function $v(t)=v_+(t)+v_I(t)+v_-(t)$ in a sum of lower, intermediate and higher modes. In particular, the equation for the most "dangerous" intermediate modes $v_I:=\Cal I_{N,K}v$ (where the parameters $N$ and $K$ will be fixed later) reads
\begin{equation}\label{2.var-I}
\Dt v_I+(1+i\omega)Av_I=\Cal I_{N,K}F'(u)(v_++v_I+v_-).
\end{equation}
As we will see later, the terms containing $v_+$ and $v_-$ are under the control due to the proper choice of the parameter $K$, so we only need to take care about the restriction of $\Cal I_{N,K}F'(u)\Cal I_{N,K}$ of $F'(u)$ to intermediate modes.
\par
We start with the term $\Cal I_{N,K} l_1(u)\Cal I_{N,K}$, see \eqref{1.derder}.
\par
\begin{lemma}\label{Lem2.sp-av} For every $\eb>0$ and $K\in\Bbb N$, there exist infinitely many values of $N\in\Bbb N$ such that, for every $u\in H$, the following estimate holds:
\begin{equation}\label{2.sp-avest}
\|\Cal I_{N,K}l_1(u)\Cal I_{N,K}v\|_{H}\le\eb\|v\|_H,\ \ v\in H.
\end{equation}
\end{lemma}
The proof of this lemma is given in \cite{FCH} and is based on the spatial averaging technique developed in \cite{M-PS88}, see also \cite{Z14}.
\par
The term containing $l_3(u)$ is estimated as follows
\begin{equation}\label{2.l3}
\|\Cal I_{N,K}l_3(u)v\|_H\le C\|\Cal I_{N,K}W(u)\|_{H}\|v\|_H\le C'(N-K)^{-s_0/2}\|v\|_H,
\end{equation}
where we have used that $W(u)$ is bounded in $H^{s_0}$. Analogously, it is not difficult to see using the assumption \eqref{1.QIMs} together with finiteness of the support of $\theta$ that the term containing $l_4(u)$ possesses the estimate
\begin{multline}\label{2.l4}
\|\Cal I_{N,K}l_4(u)v\|_H\le C(\theta(\|u\|^2_H)+|\theta'(\|u\|^2_H)|\|u\|_H)\|\Cal I_{N,K}u\|_{H}\|v\|_H\le\\\le C\((N-K)^{-1/2}+\chi_{\|P_Nu(t)\|_{H^1}\ge4\widetilde R}(t)\)\|v\|_H
\end{multline}
and, as we will see below, can be controlled with the help of an  extra operator $T'(u)v$. Thus, equation \eqref{2.var-I} can be rewritten in the form
\begin{equation}\label{2.var-I1}
\Dt v_I+(1+i\omega)Av_I-\Cal C_u(u)v_I-\Cal C_{\bar u}(u)\bar v_I =\Cal R_{I}v,
\end{equation}
where  $\Cal C_u(u):=\theta(\|u\|^2_H)a_u(W(u))$, $\Cal C_{\bar u}(u)$ is defined by \eqref{1.cbar} and
\begin{equation}
\Cal R_Iv:=\Cal I_{N,K}l_1(u)v+\Cal I_{N,K}l_3(u)v+\Cal I_{N,K}l_4(u)v-\Cal I_{N,K}T'(u)v.
\end{equation}
The main difference of equation \eqref{2.var-I1} from the analogous equations considered in \cite{FCH} is the presence of an extra non-scalar term $\Cal C_{\bar u}(u)\bar v_I$ which does not allow us to get the desired cone property in a direct way. Inspired by \cite{K20}, we will treat this term using temporal averaging technique. To this end, we write equations \eqref{2.var-I1} in Fourier base:
\begin{multline}\label{2.var-I3}
\Dt v_n+(1+i\omega)(1+|n|^2)v_n-\Cal C_u(u)v_n-\Cal C_{\bar u}(u)\bar v_n=\\=\Cal R_n v:=(\Cal R_Iv,e_n),\ \ N-K<1+|n|^2<N+K
\end{multline}
and do the following transform
\begin{equation}\label{2.av-tr}
z_n(t):=v_n(t)+\frac i{2\omega(1+|n|^2)}\Cal C_{\bar u}(u(t))\bar v_n(t).
\end{equation}
In these new variables equation \eqref{2.var-I3} reads
\begin{multline}
\Dt z_n+(1+i\omega)(1+|n|^2)z_n-\Cal C_u z_n=\Cal R_nv+\\+\frac i{2\omega(1+|n|^2)}\Cal C_{\bar u}\bar{\Cal R}_nv+\frac i{2\omega(1+|n|^2)}\Cal C_{\bar u}(\bar{\Cal C}_u-\Cal C_u)\bar v_n+\bar{\Cal C}_{\bar u}v_n)+\\+\frac i{2\omega(1+|n|^2)}(\Dt \Cal C_{\bar u})\bar v_n
 \end{multline}
On the other hand, the inverse transform to \eqref{2.av-tr} reads
\begin{equation}
v_n=\frac1{1-\frac{|\Cal C_{\bar u}|^2}{4\omega^2(1+|n|^2)^2}}\(z_n-\frac i{2\omega(1+|n|^2)}\Cal C_{\bar u}\bar z_n\).
\end{equation}
Note that we are doing this transform with intermediate modes only, so $1+|n|^2>N-K$ and, therefore, this transform is $C(N-K)^{-1}$-close to the identical one (since $\Cal C_{\bar u}$ is uniformly bounded). In a sequel, we will fix $N\gg K\gg1$, so the quantity $(N-K)^{-1}$ will be indeed small.
\par
Finally, we define the transform of the intermediate modes $v_I\to z_I$ by \eqref{2.av-tr} and
leave the higher and lower modes unchanged:
\begin{equation}\label{2.transf}
z=z_++z_I+z_-:=v_++z_I+v_-:=\Cal Q(u)v.
\end{equation}
In these new variables the equation for intermediate modes will have the form
\begin{equation}\label{2.var-I4}
\Dt z_I+(1+i\omega)Az_I-\Cal C_u(u)z_I=\Cal R_I z+\Cal R_{I,ext}z- \Cal C_{\bar u, ext}\bar z_I,
\end{equation}
where
\begin{equation}
(\Cal C_{\bar u,ext}z_I)_n=\frac {-i}{2\omega(1+|n|^2)}(\Dt \Cal C_{\bar u})\bar z_n,\ \ N-K<1+|n|^2<N+K
\end{equation}
and  the operator $\Cal R_{I,ext}$ satisfies
\begin{equation}\label{2.ext}
\|\Cal R_{I,ext}z\|_H\le C(N-K)^{-1/2}\|z\|_H
\end{equation}
with the constant $C$ which is independent of $u$, $N$ and $K$ (we have also implicitly used Lemma \ref{Lem1.temp} in order to replace $v_I$  by $z_I$ in the term containing time derivative of $\Cal C_{\bar u}$. Moreover, using the fact that all terms except of $T'(u)$ in the definition of $F'(u)$ are uniformly bounded, we replace $v_I$ by $z_I$ in the equations for higher and lower modes as well absorbing the small extra terms by the operators $\Cal R_{-,ext}$ and $\Cal R_{+,ext}$ which will
satisfy the analogue of \eqref{2.ext}.  We have also used here estimate  \eqref{1.T-est} in order to control the norm of $T'(u)(v-z)$.
\par
Comparing equations \eqref{2.var-I1} and \eqref{2.var-I4}, we see that the temporal averaging trick allows us to get rid of the problematic non-scalar term $\Cal C_{\bar u}\bar v_I$ which prevented us to use the standard invariant cone technique. The extra terms $\Cal R_{ext}z$ are not dangerous due to estimate \eqref{2.ext} and the only dangerous term is the one containing time derivative of $\Cal C_{\bar u}(u(t))$, but as we will see below, it is also under the control due to Lemma \ref{Lem1.temp} and the "good" term $T'(u)z$, so we are now ready to proceed with cone estimates.
\par
Let us define the quadratic form and the associated cone by
\begin{equation}\label{2.cone}
V(\xi):=\|Q_N\xi\|^2_H-\|P_N\xi\|^2_H \ \ \ \text{and}\ \ \Cal K^+:=\{\xi\in H,\ V(\xi)\le0\}
\end{equation}
and respectively.
\par
The following strong cone property in differential form is crucial for constructing the desired IMs.

\begin{theorem}\label{Th2.cone} It is possible to fix $K\in\Bbb N$ in such a way that
for infinitely many $N\gg K$,
any solution $u(t)$ of problem \eqref{1.cglm} satisfying \eqref{1.QIMs} and any solution $v(t)$ of the equation of variations \eqref{2.var}, the corresponding function $z(t)$ satisfies the following differential inequality:
\begin{equation}\label{2.ccone}
\frac d{dt}V(z(t))+\alpha(u(t))V(z(t))\le -\mu\|z(t)\|^2_H,
\end{equation}
where $\mu>0$ and $0<\alpha_-\le\alpha(u(t))\le\alpha_+<\infty$ for some $\alpha_\pm$ independent of~$u$.
\end{theorem}
\begin{proof} Indeed, differentiating the expression $V(z(t))$ in time and expressing time derivatives from the equations for $v(t)$ and $z(t)$, we get
\begin{multline}\label{2.cone1}
\frac12\frac d{dt}V(z(t))+\Ree((1+i\omega)Az,Q_Nz-P_Nz)-\Cal \Ree \Cal C_u(u(t))V(z(t))=\\=\Ree(l_1(u)z,Q_Nz-P_Nz)+
\Ree(l_3(u(t))z,Q_Nz-P_Nz)+\\+\Ree(l_4(u)z,Q_Nz-P_Nz)-\Ree(T'(u)z,Q_Nz-P_Nz)+\\+
\Ree(\Cal C_{\bar u}(u(t))(\bar z_++\bar z_-), Q_Nz-P_Nz)+\\
+\Ree(\Cal R_{ext}z,Q_Nz-P_Nz)-\Ree(\Cal C_{\bar u,ext}z_I,Q_Nz-P_Nz).
\end{multline}
The middle term in the LHS of this equality can be estimated as follows:
\begin{multline}
\Ree((1+i\omega)Az,Q_Nz-P_Nz)=(N+\frac12)V(z(t))+\\+((A-N-\frac12)Q_Nz,Q_Nz)+
((N+\frac12-A)P_Nz,P_Nz)\ge\\\ge (N+\frac12)V(z(t))+\frac12\|z\|^2_H+\frac12K\|\Cal P_{N,K}z\|^2_H+\\+K\|\Cal Q_{N,K}z\|^2_H+\frac12((N-A)P_Nz,z),
\end{multline}
where we have implicitly used that all eigenvalues  $\lambda_n\in\Bbb Z$.
\par
Using the fact that the operators $l_i(u)$, $i=1,3,4$, are uniformly bounded with respect to $N$ and $K$, we get
\begin{multline}
\Ree(l_i(u)z,Q_Nz-P_Nz)\le \Ree(\Cal I_{N,K}l_i(u)\Cal I_{N,K}z,Q_Nz-P_Nz)+\\+C(\|\Cal P_{N,K}z\|_H+\|\Cal Q_{N,K}z\|_H)\|z\|_H\le \|\Cal I_{N,K}l_i(u)\Cal I_{N,K}\|_{\Cal L(H,H)}\|z\|^2_H+\\+\frac1{16}\|z\|^2_H+C'(\|\Cal P_{N,K}z\|^2_H+\|\Cal Q_{N,K}z\|_H^2),
\end{multline}
where the constant $C'$ is independent of $N$ and $K$. The term containing $\Cal C_{\bar u}$ is even easier to estimate:
$$
|\Ree(\Cal C_{\bar u}(u(t))(\bar z_++\bar z_-), Q_Nz-P_Nz)|\le C\(\|\Cal P_{N,K}z\|_H^2+\|\Cal Q_{N,K}z\|^2_H\).
$$
Moreover, due to Lemma \ref{Lem1.temp}, we have
\begin{equation}\label{2.cext}
|\Ree(\Cal C_{\bar u,ext}z_I,Q_Nz-P_Nz)|\le C((N-K)^{-1/2}+\chi_{\|P_N u(t)\|_{H^1}\ge 4\widetilde R})\|z\|^2_H.
\end{equation}
 Combining the obtained estimates with \eqref{2.l3}, \eqref{2.l4} and \eqref{2.ext} and inserting the result to \eqref{2.cone1}, we arrive at
\begin{multline}\label{2.cone2}
\frac12\frac d{dt} V(z)+(N+\frac12-\Ree \Cal C_u(u(t)))V(z)+\frac12((N-A)P_Nz,z)\le\\\le (4C'-K)(\|\Cal P_{N,K}z\|^2_H+\|\Cal Q_{N,K}z\|_H^2)-\\\(\frac12\!-\!{C''}(N-K)^{-1/2}\!-\! \|\Cal I_{N,K}l_1(u)\Cal I_{N,K}\|_{\Cal L(H,H)}\!-\!C\chi_{\|P_Nu(t)\|_{H^1}\ge 4\widetilde R}\)\|z\|^2_H\\+\Ree(T'(u)P_Nz,P_Nz).
\end{multline}
According to Lemma \ref{Lem1.T}, we get
\begin{multline}\label{2.TT}
\Ree(T'(u)P_Nz,P_Nz)\le -\frac12\chi_{\|P_Nu\|_{H^1}\ge 4\widetilde R}(t)\|P_N z\|_{H^{1/2}}^2+\\+\frac12\((N^{1/2}-A^{1/2})P_Nz,z\)= \frac18K\chi_{\|P_Nu(t)\|_{H^1}\ge4\widetilde R}(t)V(z(t))-\\-\frac18K\chi_{\|P_Nu(t)\|_{H^1}\ge4\widetilde R}\|Q_Nz\|^2_H-\\-\frac12\chi_{\|P_Nu(t)\|_{H^1}\ge4\widetilde R}(t)\(\|P_Nz\|_{H^{1/2}}^2-\frac14K\|P_Nz\|_H^2\)+\\+\frac12\((N^{1/2}-A^{1/2})P_Nz,z\).
\end{multline}
Using now the elementary estimate
\begin{multline}
\|P_Nz\|_{H^{1/2}}^2-\frac14K\|P_Nz\|_H^2\ge\\\ge \((N-K)^{1/2}-\frac14K\)\|P_N\Cal I_{N,k}z\|^2_H-\frac14K\|\Cal P_{N,K}z\|^2_H
\end{multline}
together with the fact that
\begin{multline*}
\((N^{1/2}-A^{1/2})P_Nz,z\)\le\\\le
\((N^{1/2}+A^{1/2})(N^{1/2}-A^{1/2})^{1/2}P_Nz,(N^{1/2}-A^{1/2})^{1/2}P_Nz\)=\\= \((N-A)P_Nz,z\),
\end{multline*}
where we have implicitly used that $(N^{1/2}+A^{1/2})\ge1$,
and assuming that $N-K$ is big enough, we end up with the estimate
\begin{multline}
\Ree(T'(u)P_Nz,P_Nz)\le \frac18K\chi_{\|P_Nu(t)\|_{H^{1}}\ge4\widetilde R}(t)V(z(t))-\\-\frac18K\chi_{\|P_Nu(t)\|_{H^{1}}\ge4\widetilde R}(t)\|z\|^2_H+\frac38K\|\Cal P_{N,K}z\|^2_H+\frac12((N-A)P_Nz,z).
\end{multline}
Thus, fixing $K$ large enough and assuming that $K\ll N$, inequality \eqref{2.cone2} reads
\begin{multline}\label{2.cone3}
\frac12\frac d{dt} V(z)+
\(N+\frac12-\Ree \Cal C_u(u(t))-\frac18 K\chi_{\|P_Nu(t)\|_{H^1}\ge4\widetilde R}\)V(z)
\le\\\le
-\(\frac18- \|\Cal I_{N,K}l_1(u)\Cal I_{N,K}\|_{\Cal L(H,H)}\)\|z\|^2_H.
\end{multline}
Finally, fixing $N\gg1$ in such a way that the spatial averaging inequality \eqref{2.sp-avest} is satisfied with $\eb=\frac1{16}$, we get the desired cone inequality \eqref{2.ccone} with
$$
\mu:=\frac1{8},\ \ \alpha(u(t)):=2\(N+\frac12-\Ree \Cal C_u(u(t))-\frac18 K\chi_{\|P_Nu(t)\|_{H^1}\ge4\widetilde R}\)
$$
and finish the proof of the theorem.
\end{proof}
We now state the analogue of the proved theorem for finite differences between trajectories of \eqref{1.cglm}. In contrast to the case where the cones are independent of $u(t)$, this is not an immediate corollary of the proved result for infinitesimal differences and requires some additional care.
\par
Let $u_1(t)$ and $u_2(t)$ be two solutions of \eqref{1.cglm} satisfying \eqref{1.QIMs} and let $v(t):=u_1(t)-u_2(t)$. Then this function solves the following equation of variations:
\begin{equation}
\Dt v+(1+i\omega)Av=F'(u_1,u_2)v,
\end{equation}
where
\begin{equation}\label{2.F-finite}
F'(u_1(t),u_2(t)):=\int_0^1F'(s u_1(t)+(1-s)u_2(t))\,ds.
\end{equation}
The operators $l_i(u_1,u_2)$, $T'(u_1,u_2)$, $\Cal C_{\bar u}(u_1,u_2)$, $\Cal C_{\bar u,ext}(u_1,u_2)$, etc. and the transform $v\to z$ are defined analogously. The next theorem claims that the function $z(t)$ thus defined satisfies the cone inequality \eqref{2.ccone} (with probably different constants).
\begin{corollary}\label{Cor2.cone} It is possible to fix $K\in\Bbb N$ in such a way that for infinitely many $N\gg K$, any two solutions $u_1(t)$ and $u_2(t)$ of problem \eqref{1.cglm} satisfying \eqref{1.QIMs} and $v(t)=u_1(t)-u_2(t)$, the corresponding function $z(t)$ satisfies the following differential inequality:
\begin{equation}\label{2.cccone}
\frac d{dt}V(z(t))+\alpha_{u_1,u_2}(t)V(z(t))\le -\mu\|z(t)\|^2_H,
\end{equation}
where $\mu>0$ and $0<\alpha_-\le\alpha_{u_1,u_2}(t)\le\alpha_+<\infty$ for some $\alpha_\pm$ independent of~$u_1$ and $u_2$.
\end{corollary}
\begin{proof} The proof of inequality \eqref{2.cccone} repeats almost word by word the derivation of \eqref{2.ccone} given above. Indeed, as not difficult to see, the estimates for the terms contained in
 $F'(u_1,u_2)$ are analogous to the estimates for $F'(u)$. By this reason, we discuss here only the estimates containing the characteristic function $\chi_{\|P_Nu(t)\|_{H^1}\ge4\widetilde R}$ which should be naturally replaced by the sum  $\chi_{\|P_Nu_1(t)\|_{H^1}\ge4\widetilde R}+\chi_{\|P_Nu_2(t)\|_{H^1}\ge4\widetilde R}$. The possibility to do  such a replacement in the analogue of estimate \eqref{2.cext} follows from Lemma \ref{Lem1.temp1} and such a replacement in \eqref{2.l4} is straightforward, so we only need to get the analogue of \eqref{2.TT}
 for the term $T'(u_1,u_2)$. To this end we need the following elementary estimate:
 \begin{equation}\label{2.ab}
 \mu\(s\in[0,1]\,:\ |sa-(1-s)b|\ge\frac{a+b}4\)\ge \frac12,
 \end{equation}
 where $a,b>0$ and $\mu$ stands for the Lebesgue measure in $\R$. The proof of this inequality is straightforward and we left it to the reader.
 \par
 We will use this inequality with $a=\|P_Nu_1(t)\|_{H^1}$ and $b=\|P_Nu_2(t)\|_{H^1}$. Assume that
 at least one of $a,b\ge4\widetilde R$. Then, $(a+b)/4\ge\widetilde R$ and we have
 \begin{multline*}
 \mu\(s\in[0,1]\,:\ \|su_1+(1-s)u_2\|_{H^1}\ge \widetilde R\)\ge\\\ge \mu\(s\in[0,1]\,:\ |sa-(1-s)b|\ge \frac{a+b}4\)\ge\frac12
 \end{multline*}
 and, due to Lemma \ref{Lem1.T}, we arrive at
 \begin{multline}
 \Ree(T'(u_1,u_2)z,z)=\int_0^1\Ree(T'(su_1+(1-s)u_2)z,z)\,ds\le\\ -\frac1{8}\( \chi_{\|P_Nu_1(t)\|_{H^1}\ge4\widetilde R}+\chi_{\|P_Nu_2(t)\|_{H^1}\ge4\widetilde R}\)\!\|P_Nz\|^2_{H^{1/2}}+\frac12\((N-A)P_Nz,z\),
 \end{multline}
 where we have implicitly used that
 $$
 \chi_{A\cup B}(t)=\max\{\chi_A(t),\chi_B(t)\}\ge\frac12\(\chi_A(t)+\chi_B(t)\).
 $$
Thus, all necessary extensions of the estimates used in the proof of Theorem \ref{Th2.cone} are verified and the rest of the derivation of \eqref{2.cccone} is the same as in Theorem \ref{Th2.cone}. This gives us estimate \eqref{2.cccone} with $\mu=\frac18$ and
\begin{multline*}
\alpha_{u_1,u_2}(t):=2N+1-2\Ree(\Cal C_{u}(u_1(t),u_2(t))-\\-\frac1{16}K\(\chi_{\|P_Nu_1(t)\|_{H^1}\ge4\widetilde R}+\chi_{\|P_Nu_2(t)\|_{H^1}\ge4\widetilde R}\)
\end{multline*}
and finishes the proof of the corollary.
\end{proof}

 \begin{remark}\label{Rem2.equiv} Obviously estimate \eqref{2.cccone} for finite differences implies estimate  \eqref{2.ccone} for infinitesimal ones. By this reason, the infinitesimal cone property \eqref{2.ccone} holds for all values $N$ and $K$ for which \eqref{2.cccone} holds.
 \end{remark}

\section{Inertial Manifolds}\label{s3}
In this section, we construct an Inertial Manifold (IM) for the modified problem \eqref{1.cglm}. We start by recalling the definition of this object.

\begin{definition} A set $\Cal M$ is an IM for problem \eqref{1.cglm} with the base $H_+:=P_NH$ if
\par
1) $\Cal M$ is an invariant ($S(t)\Cal M=\Cal M$) Lipschiz submanifold of $H$ which is the graph of a globally Lipschitz function $M: H_+\to H_-$, i.e.
\begin{equation}\label{3.graph}
\Cal M=\{u_++M(u_+),\ \ u_+\in H_+\}
\end{equation}
\par
2) It possesses the so-called exponential tracking property, i.e. for every solution $u(t)$, $t\ge0$ of problem \eqref{1.cglm}, there exists a "trace" solution $\bar u(t)\in\Cal M$ such that
$$
\|u(t)-\bar u(t)\|_H\le Q(\|u_0\|_H)e^{-\gamma t},
$$
for some positive $\gamma$ and monotone function $Q$.
\end{definition}
It is well-known, see e.g. \cite{M-PS88,Z14} and references therein that the existence of an IM follows from the so-called cone and squeezing properties. It is also known that the cone property in differential form similar to \eqref{2.cccone} implies both cone and squeezing properties, see \cite{Z14}. Therefore, due to Theorem \ref{Th2.cone} and Corollary \ref{Cor2.cone} everything is "almost" prepared to verify the existence of the IM for our problem \eqref{1.cglm}. The only problem here is that the above mentioned theorems on the existence of IMs are usually formulated for the case where the invariant cones are independent of the point of the phase space, but in our situation only transformed function $z(t)$ satisfies the cone property $z(0)\in\Cal K^+\Rightarrow z(t)\in\Cal K^+$ and the invariant
 cones  for the associated cones for the initial infinitesimal difference $v(t)$ will depend on $u$
 through the $z\to v$ transform.  Although such dependence of invariant cones on the point of phase space is typical for the theory of normally hyperbolic invariant manifolds (at least in finite dimensions), one should be careful since some extra conditions are necessary in this case to guarantee the existence of a manifold. For instance, the classical Lorenz system possesses invariant cones, but does not possess the corresponding 2D invariant manifold. We overcome this problem by using the invariant cones for {\it finite} differences between solutions which are obtained in Corollary \ref{Cor2.cone}.
  \par
  We proceed with formulating the analogues of cone and squeezing properties for our case. For any two points $u_1$ and $u_2$ of $H$ satisfying \eqref{1.QIMs}, we define
  \begin{equation}
  \Cal K^+_{u_1,u_2}:=\{v\in H\,:\ z:=\Cal Q(u_1,u_2)v\in \Cal K^+\},
  \end{equation}
  where the cone $\Cal K^+$ and the transform $\Cal Q$ is defined by \eqref{2.cone} and \eqref{2.transf} respectively.

  \begin{lemma} Let the assumptions of Corollary \ref{Cor2.cone} hold. Then, for any two solutions $u_1(t)$ and $u_2(t)$ of problem \eqref{1.cglm} satisfying \eqref{1.QIMs} the following properties hold:
  \par
  1) Cone property: the cones $\Cal K^+_{u_1,u_2}$ are invariant, e.g.
  $$
  u_1(0)-u_2(0)\in\Cal K^+_{u_1(0),u_2(0)}\ \   \Rightarrow\ \ u_1(t)-u_2(t)\in\Cal K^+_{u_1(t),u_2(t)},\ \ t>0.
  $$
  Moreover, if in addition $u_1(t)\ne u_2(t)$, then $u_1(t)-u_2(t)\notin\partial \Cal K^+_{u_1(t),u_2(t)}$.
  \par
  2) Squeezing property: if $u_1(T)-u_2(T)\notin\Cal K^+_{u_1(T),u_2(T)}$ for some $T>0$, then
  $$
  \|u_1(t)-u_2(t)\|_H\le Ce^{-\gamma t}\|u_1(0)-u_2(0)\|_H,\ \ t\in[0,T]
  $$
  for some positive $C$ and $\gamma$ which are independent of $t$.
  \end{lemma}
Indeed, the first property is an immediate corollary of differential inequality \eqref{2.cccone} and the second one follows from the fact that the nonlinearity $F(u)$ in \eqref{1.cglm} is globally Lipschitz exactly as in \cite{Z14}.
\par
We are now ready to state the main result of this section.

\begin{theorem}\label{Th3.main} Let the assumptions of Corollary \ref{Cor2.cone} hold. Then equation \eqref{1.cglm} possesses a Lipschitz  IM $\Cal M$ with the base $H_+=P_NH$ which consists of all complete trajectories $u(t)$, $t\in\R$, for which the $Q_N$-component remains bounded as $t\to-\infty$:
\begin{equation}\label{3.rep}
\Cal M:=\left\{u(0):\ u(t), \ t\in\R \ {\rm solves\  \eqref{1.cglm}},\ \limsup_{t\to-\infty}\|Q_Nu(t)\|_H<\infty\right\}.
\end{equation}
\end{theorem}
\begin{proof} The proof of this theorem is almost identical to Theorem 2.23 of \cite{Z14}, so we only indicate the main steps of this proof for the convenience of the reader.
\par
{\it Step 1. Backward in time solutions: finite-time.} Let us fix some $T>0$ and  consider the following boundary value problem associated with \eqref{1.cglm}:
\begin{equation}\label{3.T}
\Dt u+(1+i\omega)Au=F(u),\ \ P_Nu=u_+^0,\ \ \ Q_Nu(-T)=0.
\end{equation}
We claim that this problem is uniquely solvable for any $u_+^0\in H_+$. Indeed, let us consider the map $S_T:H_+\to H_+$ defined by $S_T(u_+^T):=P_NS(T)(u_+^T,0)$, where $S(t)$ is a solution semigroup associated with equation \eqref{1.cglm}. This map is Lipschitz continuous and, moreover, due to Lemma \ref{Lem5.Q}, assumption \eqref{1.QIMs} is satisfied for any such a trajectory $u(t)=S(t+T)(u_+^T,0)$  (uniformly with respect to $T$). We also see that, for any such trajectories $u_1$ and $u_2$,
$$
u_1(-T)-u_2(-T)\in\Cal K^+_{u_1(-T),u_2(-T)}
$$
and, therefore, by the cone property, $u_1(t)-u_2(t)\in\Cal K^+_{u_1(t),u_2(t)}$ for all $t\in[-T,0]$. This, in particular, gives the injectivity of the map $S_T$ as well as the estimate
\begin{equation}\label{3.cc}
\|Q_N(u_1(t)-u_2(t))\|_H\le C\|P_N(u_1(t)-P_Nu_2(t))\|_H,
\end{equation}
due to the fact that $v$ to $z$ transform $\Cal Q(u_1,u_2)$ and its inverse are uniformly bounded.
After that, the fact that $S_T$ is invertible follows from \eqref{3.cc} and Browder open domain theorem exactly as in \cite{Z14}.
\par
{\it Step 2. Passing to the limit $T\to\infty$.}  Let us denote $u_{u_+^0}(t)$ the unique solution of problem \eqref{3.T} and define
\begin{equation}\label{3.tr}
u_{u_+^0}(t):=\lim_{T\to\infty}u_{T,u_+}(t).
\end{equation}
To verify that this limit exists, we use the squeezing property. Indeed, let $T_1>T_2$ and
$v(t):=u_{T_1,u_+^0}(t)-u_{T_2,u_+^0}(t)$. Then, since $P_Nv(0)=0$, by the cone property, we have $v(0)\notin \Cal K_{u_1(0),u_2(0)}^+$ (there is nothing to prove in the case when $Q_Nv(0)=0$) and therefore $v(t)\notin\Cal K^+_{u_1(t),u_2(t)}$. Thus, by the squeezing property,
\begin{multline}\label{3.good}
\|v(t)\|_H\le Ce^{-\gamma(t+T_2)}\|v(T_2)\|_H\le \\\le Ce^{-\gamma(t+T_2)}(\|P_Nv(T_2)\|_H+\|Q_Nv(T_2)\|_{H})\le \\\le C_1e^{-\gamma(t+T_2)}\|Q_Nv(T_2)\|_H\le C_2e^{-\gamma(t+T_2)},
\end{multline}
where we have used the fact that $u_1$ and $u_2$ satisfy $\eqref{1.QIMs}$ and the inequality $\|P_Nv(t)\|_H\le C\|Q_Nv(t)\|_H$ which follows from the cone condition. Estimate \eqref{3.good} shows that $u_{T,u_+^0}(t)$ is a Cauchy sequence in $H$ for any fixed $t$ when $T\to\infty$. This proves the existence of the limit trajectory $u_{u_+^0}(t)$.
\par
The trajectory is unique in the class of backward solutions  satisfying the determining condition:
\begin{equation}\label{3.det}
\limsup_{t\to-\infty}\|Q_Nu(t)\|_H<\infty.
\end{equation}
Indeed, let $u_1(t)$ and $u_2(t)$ be two different trajectories such that $P_Nu_1(0)=P_Nu_2(0)$ and \eqref{3.det} is satisfied for both of them. Then, obviously, $u_1(t)-u_2(t)\notin\Cal K^+_{u_1(t),u_2(t)}$ for all $t\le0$ and therefore the squeezing property gives estimate \eqref{3.good} for any $T_2>0$. Passing to the limit $T_2\to\infty$, we see that $u_1\equiv u_2$.
\par
{\it Step3. Construction of the manifold.} We define the set $\Cal M$ by \eqref{3.rep}. Then the strict invariance of it is obvious. Moreover, due to Step 2, the trajectories $u(t)$ belonging to $\Cal M$ are uniquely determined by the values $P_Nu(0):=u_+^0\in H_+$. Thus, $\Cal M$ can be presented
in  the form of \eqref{3.graph} with $M(u_+^0):=Q_Nu_{u_+^0}(0)$. Let us check that $M$ is Lipschitz continuous. Indeed, let $u_i(t):=u_{u_i^0}(t)$, $u_i^0\in H_+$, $i=1,2$ be two trajectories on the manifold $\Cal M$ and let $u_{i,T}(t)$ be their approximations on the interval $[-T,0]$ constructed in Step1. Then, obviously, $u_{1,T}(-T)-u_{2,T}(-T)\in\Cal K^+_{u_{1,T}(-T),u_{2,T}(-T)}$ and, therefore, $u_{1,T}(t)-u_{2,T}(t)\in\Cal K^+_{u_{1,T}(t),u_{2,T}(t)}$ for all $t\in[-T,0]$. Passing to the limit $T\to\infty$ and using the continuous dependence of $\Cal K^+_{u_1,u_2}$ on $u_1,u_2$, we arrive at
$$
u_{u_1^0}(t)-u_{u_2^0}(t)\in\Cal K^+_{u_{u^0_1}(t),u_{u^0_2}(t)}
$$
for all $t\le0$. In particular, this property at $t=0$ gives
\begin{multline*}
\|M(u_1^0)-M(u_2^0)\|_H=\|Q_N(u_{u_1^0}(0)-u_{u_2^0}(0))\|_H\le\\\le C\|P_N(u_{u_1^0}(0)-u_{u_2^0}(0))\|_H=C\|u_1^0-u_2^0\|_H
\end{multline*}
and the desired Lipschitz continuity is proved.
\par
{\it Step 4. Exponential tracking.} Let $u(t)$, $t\ge0$ be an arbitrary trajectory of \eqref{1.cglm}. Then, due to Lemma \ref{Lem5.Q}, we may assume without loss of generality that $u(t)$ satisfies \eqref{1.QIMs} for all $t\ge0$. Let us fix $T>0$ and consider the trajectory $\bar u_T(t)\in\Cal M$ which satisfies the condition $P_N(u(T)-\bar u_T(T))=0$. Such a trajectory exists by the construction of $\Cal M$. Then, since $u(T)-\bar u_T(T)\notin\Cal K^+_{u(T),\bar u_T(T)}$, we conclude from the squeezing property that
\begin{equation}\label{3.tr-app}
\|u(t)-\bar u(t)\|_H\le Ce^{-\gamma t}\|u(0)-\bar u_T(0)\|_H,\ \ t\in[0,T].
\end{equation}
We want to construct the desired trace $\bar u(t)$ by passing to the limit $T\to\infty$. Due to compactness arguments, it is enough to verify that the norms $\|\bar u_T(0)\|_H$ remain bounded. Indeed, due to the cone property and assumption \eqref{1.QIMs}
$$
\|u_T(0)\|_H\le\|u(0)-u_T(0)\|_H+C\le C(\|Q_N(u(0)-u_T(0))\|_H+1)\le C_1
$$
and the exponential tracking is verified, see \cite{Z14} for more details. Thus, the theorem is proved.
\end{proof}
\begin{remark}We see that, indeed, the "floating" invariant cones $\Cal K^+_{u_1,u_2}$ which depend on two points of the phase space and/or the corresponding differential cone inequality \eqref{2.cccone} allows us to verify the existence of IMs exactly in the same way as in the case of one cone independent of the points of the phase space. Crucial for this proof is the fact that the linear maps $\Cal Q(u_1,u_2)$ which reduce the floating cones to non-floating ones are uniformly bounded. This allows us to ignore this dependence in most parts of the proof given above. It would be interesting to extend the IM existence theorem to more general classes of floating cones, e.g. in the case where the projectors $P_N=P_N(u_1,u_2)$ depend also on the points of the phase space.
\par
We would like to emphasize the analogy of the representation formula \eqref{3.rep} for IMs with the analogous formula for global attractors:
\begin{equation}
\Cal A=\{u(0):\ \ u(t),\  t\in\R, \ \text{ solves \eqref{1.cglm}}\ \ \limsup_{t\to-\infty}\|u(t)\|_H<\infty\}.
\end{equation}
We do not know how deep this analogy is, but believe that it may be useful.
\end{remark}
We conclude this section by recalling that the constructed IM $\Cal M$ is actually $C^{1+\eb}$-smooth for some small positive $\eb=\eb(N)$. This fact follows from the strong cone property for equation of variations \eqref{2.var} exactly as in \cite{KZ15}. It is not difficult to see that the fact that our invariant cones are now floating does not make any essential difference in the proof given there. To get this results the assumption that
\begin{equation}\label{3.dif}
\|F(u_1)-F(u_2)-F'(u_1)(u_1-u_2)\|_H\le C\|u_1-u_2\|_{H}^{1+\eb}
\end{equation}
(which is equivalent to $C^{1+\eb}$-smoothness of $F$) is used there. In our case, $F(u)$ is only Hadamard differentiable and not even Frechet differentiable due to the structure of the cut-off function $W(u)$. This Hadamard differentiability is enough to verify the Frechet differentiability of the solution semigroup $S(t)$, but is not sufficient to get its $C^{1+\eb}$-differentiability. To overcome this difficulty, we use instead of \eqref{3.dif} its weaker version
\begin{equation}\label{3.dif-w}
\|F(u_1)-F(u_2)-F'(u_1)(u_1-u_2)\|_H\le C\|u_1-u_2\|_{H^s}^{\eb}\|u_1-u_2\|_H,
\end{equation}
where the exponent $s$ is fixed in section \ref{s1}.  This property can be easily derived from \eqref{1.W-sm}, see \cite{FCH} for more details. In the proof of differentiability of the function $M$ which generates the IM given in \cite{KZ15}, estimate \eqref{3.dif} is used in the situation when $u_1(t)$ and $u_2(t)$ are two solutions belonging to the IM and only in combination with the obvious backward Lipschitz continuity estimate
\begin{equation}\label{3.lip}
\|u_1(t)-u_2(t)\|_H\le Ce^{-K_N t}\|P_Nu_1(0)-P_Nu_2(0)\|_{H},\ \ \ t\le0.
\end{equation}
However, due to the parabolic smoothing property, estimate \eqref{3.lip} implies a stronger estimate
\begin{equation}\label{3.lips}
\|u_1(t)-u_2(t)\|_{H^s}\le C_1e^{-K_N t}\|P_Nu_1(0)-P_Nu_2(0)\|_{H},\ \ \ t\le0
\end{equation}
and by this reason there is no more difference between \eqref{3.dif} and \eqref{3.dif-w}, so the $C^{1+\eb}$-regularity of $M$ follows indeed exactly as in \cite{KZ15}, see also \cite{FCH}. By this reason, we only state the corresponding result as a theorem below leaving the detailed proof of it to the reader.
\begin{theorem}Let the assumptions of Theorem \ref{Th3.main} hold. Then the Lipschitz IM $\Cal M$ constructed there is actually $C^{1+\eb}$-smooth for some small positive $\eb=\eb(N)>0$.
\end{theorem}



 \end{document}